\documentclass[letterpaper, 10 pt, conference]{ieeeconf}  

\IEEEoverridecommandlockouts                              
\usepackage[utf8]{inputenc} 
\usepackage[T1]{fontenc}    
\usepackage{hyperref}       
\usepackage{url}            
\usepackage{booktabs}       
\usepackage{amsfonts}       
\usepackage{nicefrac}       
\usepackage{microtype}      
\usepackage{float}          

\usepackage{setspace}
\usepackage[utf8]{inputenc} 
\usepackage[T1]{fontenc}    
\usepackage{lmodern}
\usepackage{paralist}  
\usepackage{multirow}
\usepackage{tablefootnote}

\usepackage{graphicx}
\usepackage{algorithm}
\usepackage{algpseudocode}
\usepackage[utf8]{inputenc}
\usepackage[english]{babel}
\usepackage{color}

\newtheorem{theorem}{Theorem}[section]

\algnewcommand\algorithmicswitch{\textbf{switch}}
\algnewcommand\algorithmiccase{\textbf{case}}
\algnewcommand\algorithmicassert{\texttt{assert}}
\algnewcommand\Assert[1]{\State \algorithmicassert(#1)}%
\algdef{SE}[SWITCH]{Switch}{EndSwitch}[1]{\algorithmicswitch\ #1\ \algorithmicdo}{\algorithmicend\ \algorithmicswitch}%
\algdef{SE}[CASE]{Case}{EndCase}[1]{\algorithmiccase\ #1}{\algorithmicend\ \algorithmiccase}%
\algtext*{EndSwitch}%
\algtext*{EndCase}%

\newcommand{\Tt}{^{{\mbox{\tiny \bf \sf T}}}}
\def\dx{\dot{x}}
\def\dq{\dot{q}}

\newcommand{\setR}{\mathbb{R}}
\newcommand{\J}{{\cal{J}}}
\newcommand{\K}{{\cal{K}}}
\newcommand{\bigO}{{\cal{O}}}
\renewcommand{\L}{{\cal{L}}}
\newcommand{\bJ}{{\mathbf{J}}}
\newcommand{\bK}{{\mathbf{K}}}

\newenvironment{smallarray}[1]
 {\null\,\vcenter\bgroup\scriptsize
  \arraycolsep=.13885em
  \hbox\bgroup$\array{@{}#1@{}}}
 {\endarray$\egroup\egroup\,\null}

\algdef{SE}[SWITCH]{Switch}{EndSwitch}[1]{\algorithmicswitch\ #1\ \algorithmicdo}{\algorithmicend\ \algorithmicswitch}%
\algdef{SE}[CASE]{Case}{EndCase}[1]{\algorithmiccase\ #1}{\algorithmicend\ \algorithmiccase}%
\algtext*{EndSwitch}%
\algtext*{EndCase}%
\let\oldReturn\Return
\renewcommand{\Return}{\State\oldReturn}
\title{\LARGE \bf
Revised SCLP-simplex Algorithm with Application to Large-Scale Fluid Processing Networks
}

\author{Evgeny Shindin, Michael Masin, Gideon Weiss and Alexander Zadorojniy 
\thanks{Research of G. Weiss funded by
ISF Grants 249/02, 454/05, 711/09 and 286/13.  
Research of M. Masin, E. Shindin and A. Zadorojniy funded by the EU
Commission’s H2020 Program under grant agreements No 732105 and No 780788.}
\thanks{E. Shindin is with IBM Research - Haifa, Mount Carmel, Haifa, 3498825, Israel
        {\tt\small evgensh@il.ibm.com}}%
\thanks{M. Masin is with  Optibus, Eder 48a, Haifa, 3475293, Israel
        {\tt\small michael.masin@optibus.com}}%
\thanks{G. Weiss is with Department of Statistics, University of Haifa, Mount Carmel, Haifa, 3498838, Israel
        {\tt\small gweiss@stat.haifa.ac.il}}%
\thanks{A. Zadorojniy is with IBM Research - Haifa, Mount Carmel, Haifa, 3498825, Israel
        {\tt\small zalex@il.ibm.com}}%
}

\begin{document}

\maketitle
\thispagestyle{empty}
\pagestyle{empty}

\begin{abstract}

We describe an efficient implementation of a recent simplex-type algorithm for the exact solution of separated continuous linear programs, and compare it with linear programming approximation of these problems obtained via discretization of the time horizon.  The implementation overcomes many numerical pitfalls often neglected in theoretical analysis allowing better accuracy or acceleration up to several orders of magnitude both versus previous implementation of the simplex-type algorithms and versus a state-of-the-art LP solver using discretization.  Numerical study includes medium, large, and very large examples of scheduling problems and problems of control of fluid processing networks. We discuss online and offline optimization settings for various applications and outline future research directions. 

\end{abstract}

\section{Introduction} 
\label{sec.introduction}
In this paper we present an implementation and evaluate the performance of a simplex-type algorithm for the solution of a separated continuous linear programming problem (SCLP):
\begin{equation}
\begin{array}{ll}
\label{eqn.PWSCLP}
\displaystyle \max_{u(t), x(t)} & \int_0^T (\gamma\Tt + (T-t)c\Tt) u(t) + d\Tt x(t) \,dt,   \nonumber \\
 \mbox{s.t.} &  \int_0^t G\, u(s)\,ds  + F x(t) \le \alpha + a t, \\
 & \quad\; H u(t) \le b, \nonumber \\
& \quad x(t), u(t)\ge 0, \quad 0\le t \le T,  \nonumber
\end{array}
\end{equation}

SCLP problems are a special case of continuous linear programs (CLP) formulated by Bellman \cite{bellman:53}, and were first suggested for the solution of job shop scheduling problems by Anderson \cite{anderson:81}.  Many important problems can be formulated as SCLP's, but up to date these problems were always solved by discretizing the time. 
The simplex-type algorithm (SCLP-simplex)  studied here was derived by Weiss \cite{weiss:08}, and in this paper we present the first streamlined stable and efficient implementation of this algorithm and compare  its performance to the discretized LP approximation.

 SCLP-simplex  has several important advantages: 
\begin{compactitem}[-]
\item
SCLP shares many of the properties of standard LP:  it has a symmetric dual, it satisfies strong duality, solutions are obtained at extreme points, and extreme points are characterized by a combinatorial analog of basic solutions, with a well defined pivot operation.
\item
SCLP-simplex  solves SCLP in a finite number of iterations, exactly, using a parametric approach similar to Lemke's algorithm \cite{lemke-Hobson:64}.
\item
The form of the SCLP-simplex solution lends itself to perform sensitivity analysis.
\item
The SCLP-simplex  can be implemented  as a model predictive control 
for online long term optimization.
\end{compactitem}
These can be contrasted to some disadvantages of time discretization:
\begin{compactitem}[-]
\item
The time discretized LP are large and may require substantial computational effort.
\item
The solution of the discretized  approximation of SCLP may be inaccurate.  To be accurate some regions of the time horizon need only rough discretization, while others require a very fine discretization.  However  it is difficult to tell where these regions are, so the quality of the approximation is very uncertain.
\item
The structure of the optimal solution is lost in the discretized solution with many spurious basic variables that reduce the value of sensitivity  analysis, and make it practically unsuitable for model predictive control, where a new discretized problem needs to be solved from scratch in every update.
\end{compactitem}

Despite the advantages of the SCLP-simplex  of \cite{weiss:08} it has so far only 
been implemented by the author as a trial pilot aimed to verify the algorithm on very small examples.  While it received many citations referring to the theoretical results,  it has never been used in practice.  

Our contribution in this paper is a revised SCLP-simplex, that is taking advantage of several computational techniques, and is the first stable and efficient implementation of the algorithm.  
As a result,  the revised SCLP-simplex  often outperforms the discretized method, both in computation time as well as solution quality: for some large problems the discretization method fails to complete the calculation with reasonable accuracy, while SCLP-simplex solves these problems in a reasonable time, with perfect accuracy.

The rest of this paper is structured as follows:  Section \ref{sec.background} provides the structure of optimal solutions and briefly describes the SCLP-simplex algorithm.  
Section \ref{sec.implementation} discusses details of our revised implementation. Section \ref{sec.setup} describes the potential applications and sets the experimental setup, followed by a computational study in Section \ref{sec.results}.  
Section \ref{sec.impact} summarizes the work and discusses future directions.

\section{Background on SCLP}
\label{sec.background}
The main relevant SCLP references to our approach are \cite{bellman:53,anderson:81,anderson-nash:87,pullan:93,pullan:96duality,pullan:95forms,pullan:00convergence}.  Additional approaches are shown in
\cite{luo-bertsimas:98,fleischer-sethuraman:05,bampou-kuhn:12}.  The SCLP-simplex algorithm is based on \cite{weiss:08}, several extensions and generalizations are shown in \cite{wang-yao-zhang:09,shapiro:01,shindin-weiss:14duality,shindin-weiss:15form,shindin-weiss:18simplex}. 
In this section we describe the structure of optimal solutions
%

In (\ref{eqn.PWSCLP}), the matrices $G,H,F$ are $K\times J$, $I\times J$, $K\times L$ dimensional respectively, and we number the slacks $x_1,\ldots,x_K$, and $u_{J+1},\ldots,u_{J+I}$. Denote by  $\bK = (1,\ldots,K+L)$  the indexes of the primal state variables $x_k(t)$ and by $\bJ=(1,\ldots,J+I)$  the  indexes of the primal control variables $u_j(t)$.  
The symmetric dual to (\ref{eqn.PWSCLP}) is 
\begin{equation}
\begin{array}{ll}
\label{eqn.DWSCLP1}
\displaystyle \min_{p(t),q(t)} & \int_0^T (\alpha\Tt + (T-t)a\Tt) p(t) + b\Tt q(t) \,dt,    \nonumber  \\
\mbox{s.t.} &  \int_0^t  G\Tt\, p(s)\,ds + H\Tt q(t) \ge \gamma + c t, \\
 & \quad\; F\Tt p(t) \ge d, \nonumber \\
& \quad q(t), p(t)\ge 0, \quad 0\le t \le T,   \nonumber
\end{array}
\end{equation}
with dual state variables, including slacks,  $q_j(t),\,j\in \bJ$ and dual control variables $p_k(t),\,k\in\bK$.  
Note that {\em the dual problem runs in reversed time.} Complementary slackness is defined by:
\begin{equation}
\label{eqn.compslack}
\textstyle
\int_0^T x(t)\Tt p(T-t) dt = \int_0^T u(t)\Tt q(T-t) dt=0
\end{equation}

Under easily checked feasibility and boundedness conditions, and under  non-degeneracy,  SCLP has a unique strongly dual solution. 
The optimal solution has   piecewise constant primal and dual controls and continuous piecewise linear primal and dual state variables,   with breakpoints $0=t_0<t_1<\cdots <t_N=T$.   The solution is then fully described by the breakpoints, by the initial state values $x(0)=x^0,\,q(0)=q^N$,  and by the values of the controls and of the derivatives of the states $u_j^n=u_j(t)\,,p_k^n=p_k(T-t)$, $\dx_k^n=\dx_k(t),\,\dq_j^n=\dq_j(T-t)$ for $t_{n-1}<t<t_n,\,n=1,\ldots,N$.
The values of the primal and dual states at the breakpoints are $x_k^n=x_k(t_n),\,q_j^n=q_j(T-t_n),\,n=0,\ldots,N$.
The initial values, $x^0,q^N$, are optimal solutions of the {\em Boundary-LP}:
\begin{equation}
\label{eqn.boundary}
\begin{array}{lll}
\max \quad    [0 \; d\Tt] x^0,  &\quad&  \min  \quad   [b\Tt \; 0] q^N,   \\
\mbox{s.t.}  \quad     [I \; F] x^0 = \alpha,  &&  \mbox{s.t.}  \quad     [H\Tt  -I] q^N = \gamma,   \\
\qquad \quad  x^0 \ge 0,   &&   \quad \qquad  q^N\ge 0.  \\
\end{array}
\end{equation}
with $\K_0$, $\J_{N+1}$ the indexes of the basic variables $x_k^0, q_j^N$.
Note:  the LP for $x^0,q^N$ does not involve $T$, so the boundary values are the same for all time horizons.

Values of the controls and slopes of states in the intervals are complementary slack basic solutions of the primal and dual {\em Rates-LP}$(\K,\J)$:
 \begin{equation}
\label{eqn.rates1}
\arraycolsep=2pt
\begin{array}{cc}
\begin{array}{ll}
\max &  [c\Tt\;0] u +  [0\;d\Tt] \dx    \\
\mbox{s.t.}& [G\;0] u + [I\;F] \dx = a,   \\
  &   [H\;I] u \quad \quad \quad  = b,  
\end{array}
&\quad
\begin{array}{l}
 \dx_k \in \setR \;\forall k \in \K,\\  \dx_k \in \setR^+ \;\forall k \notin \K, \\
 u_j = 0 \;\forall j \in \J,\\  u_j \in \setR^+ \;\forall j \notin \J,
 \end{array}
 \end{array}
\end{equation}
 \begin{equation}
\label{eqn.rates2}
\arraycolsep=2pt
\begin{array}{cc}
\begin{array}{cl}
\min&   [a\Tt\;0] p + [0\;b\Tt] \dq   \\
\mbox{s.t.}&  [G\Tt\;0] p + [H\Tt\;\mbox{-}I]\dq = c ,   \\
&   [F\Tt\;{\mbox{-}I}] p \quad \quad \quad   = d,
\end{array}
&\quad
\begin{array}{l}
\dq_j \in \setR \;\forall j \in \J,\\  \dq_j \in \setR^+ \;\forall j \notin \J, \\
p_k = 0 \;\forall k \in \K,\\  p_k \in \setR^+\; \forall k \notin \K,  
\end{array}
\end{array}
\end{equation}
where for interval $(t_{n-1},t_n)$ the primal basis is $B_n=\{u^n_j,\dx^n_k :  j\not\in\J_n,k\in\K_n\}$
with complementary dual basis $B^*_n =\{p^n_k,\dq^n_j : k\not\in\K_n,j\in\J_n\}$.

The bases 
have the following properties:  
\newline -- Compatibility to the boundary:  $\K_0 \subseteq \K_1$, $\J_{N+1} \subseteq \J_N$.
\newline -- Adjacency:  $B_n,B_{n+1}$ are adjacent:  in the pivot $B_n \to B_{n+1}$ a single basic variable $v^n$ leaves the basis and a single basic variable $w^n$ enters.

The breakpoints $t_1,\ldots,t_{N-1}$ are determined by the following equations for the interval lengths $\tau_n=t_n - t_{n\mbox{-}1}$:
\begin{equation}
\label{eqn.breakponts}
\arraycolsep=2pt
\begin{array}{ll} 
x_k(t_n) =  x_k^0 + \sum_{m=1}^n \dx_k^m \tau_m = 0  &  \mbox{if $v^n=\dx_k$}, \\
q_j(T-t_n) = q_j^N + \sum_{m=N}^{n+1}  \dq_j^m \tau_m =0 &     \mbox{if $v^n=u_j$}, \\
\tau_1+\cdots+\tau_N = T. &
\end{array}
\end{equation}
The remaining values are determined by:
\begin{equation}
\label{eqn.othervalues}
\begin{array}{l} 
x_k(t_n) =  x_k^0 + \sum_{m=1}^n \dx_k^m \tau_m,  \\
q_j(T-t_n) = q_j^N + \sum_{m=N}^{n+1}  \dq_j^m \tau_m. 
\end{array}
\end{equation}
Given a sequence of adjacent bases $B_1,\ldots,B_N$ we can calculate all the controls and slopes of states, the breakpoints, and the values of the primal and dual states at all breakpoints. 
It is an optimal base sequence if: 
\begin{theorem}[\cite{weiss:08}]
\label{thm.optimal}
If a sequence of bases $B_1,\ldots,B_N$ are compatible with $\K_0,\J_{N+1}$ and are adjacent, and if all the values of the primal and dual state variables and the interval lengths determined by equations (\ref{eqn.boundary})--(\ref{eqn.othervalues})  are positive, then this  is an optimal solution of the SCLP.  
\end{theorem}

The SCLP-simplex algorithm is similar to the parametric self dual simplex algorithm, also known as Lemke's  \cite{lemke-Hobson:64} 
algorithm, for the solution of standard LP.
In Lemke's algorithm a pair of dual LP's is solved parametrically, starting from an objective of $-1$'s and a r.h.s. of 1's with the trivial optimal solution where the  primal and dual basic variables   are the slacks.  Then it solves all the LP's along the parametric line 
$\L(\theta)=(1-\theta)\left[\begin{smallarray}{c} -1 \\ 1 \end{smallarray}\right] + 
\theta \left[\begin{smallarray}{c} c \\ b \end{smallarray}\right]$.  
The solution partitions $0<\theta_1 <\cdots <\theta_M =1$, and at each $\theta_\ell$ either a primal or a dual variable shrinks to 0, and a single pivot or several pivots are performed to obtain the optimal basis for $\theta > \theta_\ell$.    

The SCLP-simplex is initiated by solving (\ref{eqn.boundary}) for $x(0),q(0)$ and obtaining $B_1$, the optimal basis for 
Rates-LP$(\K_0,\J_{N+1})$.  $B_1$ is the initial optimal base sequence for small time horizons.    
Then, in analogy to Lemke's  algorithm, SCLP-simplex solves SCLP parametrically, by increasing the time horizon $\theta T$ over $0 < \theta \le 1$, with iterations needed at $0<\theta_1 <\cdots <\theta_M =1$.
 In each validity range $\theta_{\ell-1} < \theta < \theta_\ell$,  the optimal solution is defined by an optimal base sequence $B_1,\ldots,B_N$,  and the 
  values of $x^n,q^n,\tau_n$ are affine functions of $\theta$ with well defined derivatives $\delta(\cdot) = \frac{d(\cdot)}{d\theta}$. 
  At $\theta_\ell$ {\em a collision} occurs:  either one or several intervals or else a primal or a dual state variable shrink to zero at a breakpoint $t_n$, and an SCLP-pivot is performed to obtain the optimal base sequence in the next validity range.
  
The main difference is that in Lemke's LP   at each step, only the one basis defines the optimal solution, while in  SCLP-simplex  the optimal solution consists of the base sequence.  
Each iteration of SCLP-simplex consists of two steps: calculation of the validity range, and  SCLP-pivot.
{\newline -- \em Calculation of the validity range},  $\theta_\ell$
\begin{equation}
\label{eqn.parameterlimit}
\textstyle
\theta_\ell = \theta_{\ell-1} +  \inf_{\delta(\cdot) < 0} \left\lbrace\theta : -\frac{\tau_n}{\delta \tau_n},   
 -\frac{x^n_k}{\delta x^n_k},  -\frac{q^n_j}{\delta q^n_j}\right\rbrace.  
\end{equation}
We then obtain the type of collision, interval shrinking or state variable shrinking, the location of the collision, $t_n$, and the bases on both sides to the collision, $B',B''$.
{\newline -- \em SCLP-pivot},  consists of the the following:
\begin{compactitem}[-]
\item
 If intervals shrunk to 0, remove bases between $B',B''$.   If $B', B''$ are adjacent, you found the new base sequence.
 \item
 Otherwise, $B'\setminus B'' = \{v_1,v_2\}$, choose proper $v',v'' \in  \{v_1,v_2\}$, and solve Rates-LP$(\K^*, \J^*)$, where $\K^* = \{k: \dx_k \in B'\}\setminus v'', \J^* = \{j: u_j \not\in B''\}\cup v'$  to obtain basis $D$.  If $D$ is adjacent to $B',B''$, insert it between $B',B''$,  you found the new base sequence.
 \item
Otherwise, formulate a subproblem, which is an SCLP of smaller dimension, and solve it to obtain an optimal base sequence $D_1,\ldots,D_L$, and insert it between $B',B''$,
you found the new base sequence.
  \end{compactitem}
 For collisions at 0 or $T$ the steps of the pivot are slightly modified.
\section{Revised SCLP-simplex Algorithm}
\label{sec.implementation}
SCLP-simplex \cite{weiss:08} will always work perfectly under the following conditions:  The problem needs to be non-degenerate, and all calculations need to be done with perfect accuracy.  
However,  previous implementation, that was intended only as a pilot for concept verification, used Matlab with  floating point  calculations
and was vulnerable to  degeneracy and inaccuracies.  
Moreover straightforward implementation suffered from memory and performance issues, that substantially slowed down the algorithm. As a result only problems with $K+L+J+I \le 100$ could be solved before the program crashed or ran out of time.

To improve performance and numerical stability of the SCLP-simplex algorithm we thoroughly analyzed each  step and  developed the 
{\em revised SCLP-simplex} algorithm. Python implementation of the algorithm is available  at GitHub \url{https://github.com/IBM/SCLPsolver}.

The following analysis and implementation enhancements led to substantial performance gains.
\begin{compactitem}[-]
\item
Base sequence representation is one of the problematic points of the SCLP-simplex: the  choice to store only the  indexes of the basic variables, and re-solve all Rates-LP fresh at each iteration requires an impractical  amount of computations, and storage of the simplex dictionaries for all bases involves memory issues, in storage and in updating. Since all bases are adjacent our algorithm stores simplex dictionaries only for some of them, keeping the list of pivots between all bases.
This requires more computations when a new basis ($D$) needs to be calculated, but drastically reduces the required memory.  The code obtains the available RAM and adjusts the number of stored dictionaries accordingly, maintaining  evenly spaced dictionaries.
\item
Values of all $\dx, \dq$ are kept since only a small part of them is updated during the SCLP pivot.
\item
Equations (\ref{eqn.breakponts}) are re-structured to increase sparsity. The resulting system is solved for $\delta \tau(\theta_\ell)$ using LU factorization, while $\tau(\theta_\ell)$ is calculated as $\tau(\theta_{\ell-1}) + \delta \tau(\theta_{\ell-1})$.
In a case when exactly one interval shrinks to $0$ and is replaced by a single interval, the corresponding SCLP pivot changes only one column of coefficients in (\ref{eqn.breakponts}).
In this case we use product form of inverse (PFI) to calculate $\delta \tau(\theta_\ell)$ from $\delta \tau(\theta_{\ell-1)}$ using LU decomposition obtained in the previous iteration. Such design reduces the number of operations from $\bigO(N^3)$ to $\bigO(N)$ complexity.
\item
To evaluate $\theta_\ell - \theta_{\ell-1}$ in (\ref{eqn.parameterlimit}) we use only the values of $\delta \tau_n, \tau_n$ and of $x^n, q^n$ at local minima where $\delta x^n,\delta q^n$ are negative. This allows us to solve only part of  (\ref{eqn.othervalues}). For this purpose we keep a list of all the  local minima of $x(t),q(t)$.  Updating this list  involves  only small changes in each iteration.
\item
The basis $D$ is calculated by pivoting the simplex dictionary of $B'$ or $B''$ using non-standard pivoting rules, where entering and leaving variables are determined by $B'' \setminus B'$ and $B' \setminus B''$ and by the type of the collision.  If and optimal $D$ which is adjacent to both $B', B''$ exists, it is always found  by this  single pivot.
\item
To calculate $B'' \setminus B'$ and $B' \setminus B''$ we use the list of pivots, and avoid computationally expensive set differences.
\item
The subproblems are solved with reduced dimension, and are therefore quite small.
\end{compactitem}

Numerical pitfalls lead to poor stability of the naive implementation that depends on non-degeneracy of all dictionaries, and on correct identification of  the collision types.  In theory, perturbation of the data, in particular of $a$ and $c$ can achieve all the required non-degeneracy, and ensure unique execution of all the steps of the algorithm.  However, in practice, accumulation of numerical errors due to floating point operations may impede such clean runs.
In most cases numerical problems arise when we should decide if a value is $0$ or just a small floating point number or when we should compare close floating point numbers to determine the sequence of intervals that shrink to $0$.  
Numerical instabilities lead to incorrect collision classification which may not be recognized at the classification time, but will appear in further steps.  This creates situations that are impossible in  theory, but do occur in practice, such as
\begin{compactitem}[-]
\item
Incorrect collision: $\tau_{n},\dots,\tau_{n'}$ shrink to $0$ in the middle of the base sequence, but $|B_{n-1} \setminus B_{n'+1}| > 2$.
\item
Incorrect pivot: 
the new basis $D$ is adjacent to $B'$ and $B''$ but values of $\dx_k$ and/or $\dq_j$ for this basis leading to a discontinuity in the state variables.
\item
Incorrect subproblem formulation: during the solution of a  subproblem the base sequence $D_1,\dots, D_L$ arrives at an infeasible or unbounded basis $D_l$, or the subproblem parametric line reaches $\theta=1$, but $D_1$ and $B'$ or $D_L$ and $B''$ are not adjacent.
\item
Zero lengths interval shrink: a new interval obtained at the previous iteration shrinks, introducing infinite loop, where the parametric line is not moving forward, or a number of zero lengths intervals shrink, which impedes the collision classification. 
\end{compactitem}
In these cases we return to the classification step, then change the numerical tolerances and reclassify the collision. Once the problem is resolved, tolerances are readjusted. 
In addition, starting from an  iteration where the collision classification was not clear, we store information that allows us to go back. If the reclassification fails, we go to an earlier iteration trying to reclassify the collision there.

\section{Applications and Experimental Setup}
\label{sec.setup}

The original motivation for SCLP was to define tractable optimization models for the {\em job shop scheduling problems} \cite{anderson:81}. For example, {\em micro-chips wafer fabricalion:}  starting  as a wafer of pure silicon crystal, the wafer undergoes up to 1200 operations revisiting a set of up to 60 workstations, in a re-entrant line production process, to produce several hundred computer chips, in a cycle time of some six weeks.  The problem is to control the movement of some 60,000 wafers over a time horizon of 6 weeks, with work in process value of $200\times 10^6$\$.  A rich literature on this problem includes \cite{vanZant:13,chen-etal:88,wein:88,kumar:93}.  Specifically, solution of SCLP  to these problems is described in \cite{luo-bertsimas:98,bertsimas-nasrabadi-paschalidis:14} with small to medium size examples. This application motivates our first experiment solving problems of the full size.
\newline -- {\em Transient control of multi-class queueing networks:}  
Items of several types arrive at the system, and need to follow individual paths through various service stations, and we need to control admissions, routing and sequencing items throughout the system \cite{harrison:88,wein:92,kelly-laws:93,dai:95,bramson:08,meyn:08}.  In the second experiment we test the revised SCLP-algorithm to solve the SCLP formulation of this problem enabling asymptotically optimal control of the stochastic system \cite{nazarathy-weiss:09}.

\cite{bertsimas-nasrabadi-paschalidis:14} have shown that {\em robust optimization} of both problems can also be formulated as SCLP. The formulation increases the problem sizes even further. Moreover, many additional application benefit from the proposed algorithm, e.g., {\em health systems}, where the flow of patients through emergency rooms, hospitalization, operating theaters, requires concerted use of resources, and patients follow complex paths through the system, see \cite{GansKM:03,mandelbaum-momcilovic:12}; {\em the quickest evacuation problem} for evacuation of stadiums, convention centers, amusement parks; and even the standard {\em maximum flow problem} over time usually solved via discretization of time can be easily solved as an SCLP.  

We chose two types of SCLP problems to evaluate the performance of our revised SCLP-simplex:  SCLP for a re-entrant line that approximates the wafer fab job shop scheduling problem, and SCLP for transient control of a multi-class queueing network, that approximates a stochastic service system.

{\em The re-entrant line SCLP:}  
The  state variable $x_k(t)$ is the quantity of wafers waiting to complete production step $k$, also referred to as contents of buffer $k$. 
The control variable $u_k(t)$  is the processing capacity allocated to buffer $k$.
The objective is:
\[
\textstyle
\min_{u(t),x(t)} V =  \int_0^T ( h\Tt x(t) + g\Tt u(t)) dt 
\]
where $h$ is the vector of holding costs, and $g$ the vector of operating costs.  
$G=\left[\begin{smallarray}
{ccccc} 
1  & 0 & 0 & \ldots & 0 \\
-1 & 1 & 0 & \ldots & 0 \\
0 & -1 & 1 & \ldots & 0 \\
\vdots & \ddots& \ddots &\ddots  & \vdots \\
0 &\ldots &0 & -1 & 1 
 \end{smallarray}  \right]$
represents the re-entrant line, with items moving from buffer 1 to 2 etc, until  $K$ and  out.
$H$  is the resource constraints matrix where rows $i=1,\ldots,I$ represent workstations, and $H_{i,k} = m_k$ where $m_k$ is the processing time per item at production step $k$ performed at workstation $i$.
The vector $b$ is all 1's, as capacity of workstation $i$.
Initial inventory in buffer $k$ is $\alpha_k$,   $a_k$ is the rate of   wafers input to buffer $k$, often $a_1>0$, and all other $a_k=0$.
To conform to (\ref{eqn.PWSCLP}) we substitute  $x(t) = \alpha+a t - \int_0^t Gu(s)ds$ in the objective to obtain a surrogate equivalent maximization objective, where $\gamma = -g$ and $c= h\Tt G$.  For this problem $F$ and $d$ of (\ref{eqn.PWSCLP}) are nil.

{\em Transient control MCQN:}   queues $k=1,\ldots,K$ have $x_k(t)$  discrete  items waiting for processing.  Activity $j=1,\ldots,J$ will process an item from $k=k(j)$ for a random time with mean $m_j$, and then route the completed item to queue $l$ with probability $p^j_{k(j),l}$, or send it out of the system with probability $1-\sum_l p^j_{k(j),l}$.  To do so it will use workstation $i=s(j)$.   Items arrive at queue $k$ in a random stream, at rate $a_k$.   We wish to decide which activities to employ at each time unit, so as to minimize  the expected sojourn times  or equivalently 
the (weighted) queue lengths,  starting from some initial $x_k(0)$, over finite time horizon $[0,T]$.  The SCLP approximation has 
\begin{compactitem}[$-$]
\item
$G$ is $K\times J$ with elements:  $G_{l,j}=\left\{\begin{smallarray}{ll} 1, &\; l=k(j) \\  -p^j_{k(j),l}, &\; l \ne k(j)  \end{smallarray} \right.$.
\item
$H$ is $I\times J$ with elements $H_{i,j} =  \left\{\begin{smallarray}{ll} m_j, &\; s(j)=i \\  0,  &\;  s(j) \ne i   \end{smallarray} \right.$.
\end{compactitem}

\section{Computational Results} 
\label{sec.results}

We compare our implementation of SCLP-simplex vs. LP discretization with up to 1000 time intervals. 
for the  two problem classes presented in Section \ref{sec.setup}. 
For each class we defined five settings with different amounts of servers, buffers, and job classes and randomly generated 10 problems. Randomly generated parameters include initial fluids, arrival rates, processing rates and time horizons. Their probability distributions, were chosen to avoid trivial or degenerate situation.
All experiments were performed on Lenovo ThinkPad W541 notebook computer with Intel Core i7-4810MQ processor and 16GB RAM running Windows 10. SCLP-simplex ran on Python 3.7.7 with NumPy 1.18.1, linked to Intel MKL. Some vector and matrix operations parallelized using Cython and Intel OpenMP library. The discretized problems were solved by IBM Ilog Cplex Optimization Studio 12.10 using a barrier algorithm that showed the best performance for these problems. Both algorithms ran on eight cores.

\subsection*{LP discretization}
Naive discretization may produce a quadratic number of unnecessary non-zero coefficients in the LP problem. Here we show a much more efficient discretization method  to obtain an approximate solution of SCLP by regular LP solvers.

We consider a uniform time partition $0=t_0, \dots, t_N=T$, where $N$ is the number of intervals and for $n=1,\dots,N:$ $t_n = t_{n-1} + \tau, \tau = T/N$.
For each time interval $n=1,\dots,N$ we define a vector of discrete controls $u[n]$ and state variables $x[n] $.  Then the discretizition of problem (\ref{eqn.PWSCLP}) without $F$ and $d$ can be represented as:
\begin{equation}
\label{eqn.PWSCLPredDiscr1}
\arraycolsep=2pt
\begin{array}{ll}
\displaystyle \min_{u,x}  & V_{LP} = \sum_{n=1}^N (\tau g\Tt u[n] + 0.5\, h\Tt (x[n] + x[n \mbox{-} 1]) 
\nonumber  \\
s.t.&  \tau G\, u[n] + x[n] - x[n \mbox{-} 1] = a \tau \quad \forall n \label{ctrl2} \\
&\mbox{with} \qquad  x[0] =   \alpha         \nonumber \\
& H \cdot u[n] \le b,  \qquad u[n],x[n]\ge 0, \quad \forall n.  \nonumber
\end{array}
\end{equation}

\subsection*{Results}
The objective value $V$ provided by SCLP-simplex is the accurate theoretical minimum, the   {\em relative error} is measured as $(V_{LP} - V) / V$. To compare performance, we measure {\em relative time} as the ratio of  run times of CPLEX over SCLP-simplex. 
The results for  re-entrant line and NCQN problems are shown in Table~\ref{table:reentrant}, Fig. \ref{fig:reentrant_time}, \ref{fig:reentrant_error} and Table~\ref{table:MCQN}, Fig. \ref{fig:MCQN_time}, \ref{fig:MCQN_error}, respectively. 
\begin{table}[H]
\caption{Results on re-entrant line problems}
\label{table:reentrant}
\scalebox{0.68}{
\begin{tabular}{@{\hspace{0.1cm}}@{\hspace{0.1cm}}r@{\hspace{0.1cm}}@{\hspace{0.1cm}}r@{\hspace{0.1cm}}@{\hspace{0.1cm}}r@{\hspace{0.1cm}}@{\hspace{0.1cm}}r@{\hspace{0.1cm}}@{\hspace{0.1cm}}r@{\hspace{0.1cm}}@{\hspace{0.1cm}}r@{\hspace{0.1cm}}@{\hspace{0.1cm}}r@{\hspace{0.1cm}}@{\hspace{0.1cm}}r@{\hspace{0.1cm}}@{\hspace{0.1cm}}r@{\hspace{0.1cm}}@{\hspace{0.1cm}}r@{\hspace{0.1cm}}@{\hspace{0.1cm}}r@{\hspace{0.1cm}}@{\hspace{0.1cm}}r@{\hspace{0.1cm}}}
\hline
\multirow{3}{*}[-15pt]{\rotatebox[origin=c]{90}{Servers}} & \multirow{3}{*}[-15pt]{\rotatebox[origin=c]{90}{Buffers}} & \multirow{3}{*}{\rotatebox[origin=c]{90}{Time horizon }}  & \multicolumn{3}{@{\hspace{0.1cm}}c@{\hspace{0.1cm}}}{Average numbers} &  \multicolumn{6}{@{\hspace{0.1cm}}c@{\hspace{0.1cm}}}{Discretization} \\ \cline{4-12} 
 &  &  & \multirow{3}{*}{\parbox{0.9cm}{\centering Run time, sec}} & \multirow{3}{*}{Steps} & \multirow{3}{*}[-5pt]{\rotatebox[origin=c]{90}{Intervals }} &  \multicolumn{2}{{@{\hspace{0.1cm}}c@{\hspace{0.1cm}}}}{10} & \multicolumn{2}{{@{\hspace{0.1cm}}c@{\hspace{0.1cm}}}}{100} & \multicolumn{2}{{@{\hspace{0.1cm}}c@{\hspace{0.1cm}}}}{1000} \\ \cline{7-12}\noalign{\smallskip}
 &  &  &  &  &  & \multicolumn{2}{{@{\hspace{0.1cm}}c@{\hspace{0.1cm}}}}{Relative:} & \multicolumn{2}{{@{\hspace{0.1cm}}c@{\hspace{0.1cm}}}}{Relative:} & \multicolumn{2}{{@{\hspace{0.1cm}}c@{\hspace{0.1cm}}}}{Relative:} \\
 &  &  &   &  &  & error & time & error & time & error & time \\[0.55cm] \hline
20 & 400 & 600 & 3.352 & 921.7 & 440.8 & 4.87 & 0.0625 & 0.400 & 0.763 & 0.022 & 12.449\\ 
30 & 600 & 900 & 7.842 & 2244.7 & 667.8 &  11.36 & 0.0289 & 1.013 & 0.336 & 0.067 & 7.105\\ 
40 & 800 & 1200 & 16.570 & 3109.3 & 883.8 & 13.54 & 0.0165 & 1.236 & 0.214 & 0.090 & 4.454\\ 
50 & 1000 & 1500 & 34.753 & 4402.3 & 1113.4 & 19.97 & 0.0099 & 1.866 & 0.125 & 0.146 & 2.429\\ 
60 & 1200 & 1800 & 67.419 & 5699.0 & 1322.9 & 23.10 & 0.0066 & 2.183 & 0.081 & 0.176 & 1.617\\    \hline
\end{tabular}
}
\end{table}
\begin{figure}[H]
    \centering
    \begin{minipage}{0.47\columnwidth}
        \centering
        \includegraphics[scale=0.37]{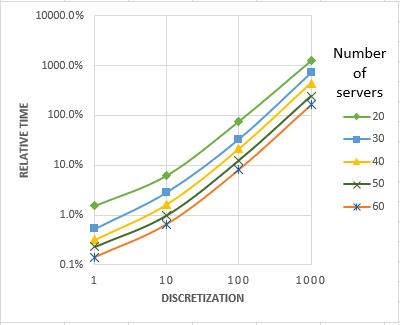}
  	\caption{Re-entrant line discretization relative time}
  	\label{fig:reentrant_time}
    \end{minipage}\hfill
    \begin{minipage}{0.47\columnwidth}
        \centering
       \includegraphics[scale=0.37]{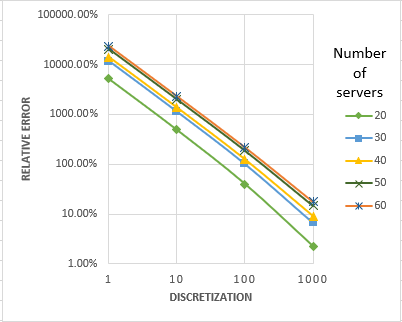}
  		\caption{Re-entrant line discretization relative error}
  		\label{fig:reentrant_error}
    \end{minipage}
\end{figure}
  For both problems, low discretization leads to non-optimal solutions with large relative errors while high discretization is resource-thirsty with long run times. We remark, that Cplex provide almost constant 100\% load on all processor cores, while for the revised SCLP-simplex implementation the load of all cores is not constant with 50\% average load for all cores.  
 This indicates that we may be able to further improve SCLP-simplex by exploiting more parallelization.
\begin{table}[H]
\caption{Results on MCQN problems}
\label{table:MCQN}
\scalebox{0.68}{
\begin{tabular}{@{\hspace{0.1cm}}@{\hspace{0.1cm}}r@{\hspace{0.1cm}}@{\hspace{0.1cm}}r@{\hspace{0.1cm}}@{\hspace{0.1cm}}r@{\hspace{0.1cm}}@{\hspace{0.1cm}}r@{\hspace{0.1cm}}@{\hspace{0.1cm}}r@{\hspace{0.1cm}}@{\hspace{0.1cm}}r@{\hspace{0.1cm}}@{\hspace{0.1cm}}r@{\hspace{0.1cm}}@{\hspace{0.1cm}}r@{\hspace{0.1cm}}@{\hspace{0.1cm}}r@{\hspace{0.1cm}}@{\hspace{0.1cm}}r@{\hspace{0.1cm}}@{\hspace{0.1cm}}r@{\hspace{0.1cm}}@{\hspace{0.1cm}}r@{\hspace{0.1cm}}}
\hline
\multirow{3}{*}[-15pt]{\rotatebox[origin=c]{90}{Servers}} & \multirow{3}{*}[-15pt]{\rotatebox[origin=c]{90}{Buffers}} & \multirow{3}{*}{\rotatebox[origin=c]{90}{Time horizon }}  & \multicolumn{3}{@{\hspace{0.1cm}}c@{\hspace{0.1cm}}}{Average numbers} &  \multicolumn{6}{@{\hspace{0.1cm}}c@{\hspace{0.1cm}}}{Discretization} \\ \cline{4-12} 
 &  &  & \multirow{3}{*}{\parbox{0.9cm}{\centering Run time, sec}} & \multirow{3}{*}{Steps} & \multirow{3}{*}[-5pt]{\rotatebox[origin=c]{90}{Intervals }} &  \multicolumn{2}{{@{\hspace{0.1cm}}c@{\hspace{0.1cm}}}}{10} & \multicolumn{2}{{@{\hspace{0.1cm}}c@{\hspace{0.1cm}}}}{100} & \multicolumn{2}{{@{\hspace{0.1cm}}c@{\hspace{0.1cm}}}}{1000} \\ \cline{7-12}\noalign{\smallskip}
 &  &  &  &  &  & \multicolumn{2}{{@{\hspace{0.1cm}}c@{\hspace{0.1cm}}}}{Relative:} & \multicolumn{2}{{@{\hspace{0.1cm}}c@{\hspace{0.1cm}}}}{Relative:} & \multicolumn{2}{{@{\hspace{0.1cm}}c@{\hspace{0.1cm}}}}{Relative:} \\
 &  &  &   &  &  & error & time & error & time & error & time \\[0.55cm] \hline
20 & 200 & 100 & 0.988 & 663 & 271 & 1.244 & 1.027 & 0.049 & 12.494 & 0.0013 & 212.83\\ 
40 & 400 & 100 & 4.817 & 1887 & 535 & 0.893 & 1.395 & 0.034 & 18.770 & 0.0009 & 448.42\\ 
60 & 600 & 100 & 17.098 & 3899 & 815 &  0.913 & 1.242 & 0.036 & 19.992 & 0.0009 & 946.59 \tablefootnote{Solution of one of the problems with 60 servers and 600 buffers took 131211 sec. Excluding this problem the average relative time becomes 484.46. Discretization to 1000 intervals is currently infeasible for larger MCQN problems.}\\ 
80 & 800 & 100 & 41.655 & 6424 & 1080 & 1.017 & 1.217 & 0.039 & 23.967 &  & \\ 
100 & 1000 & 100 & 91.809 & 9466 & 1356 & 0.922 & 1.051 & 0.036 & 15.809 &  & \\ 
\hline
\end{tabular}
}
\end{table}
\begin{figure}[H]
    \centering
    \begin{minipage}{0.47\columnwidth}
        \centering
         \includegraphics[scale=0.37]{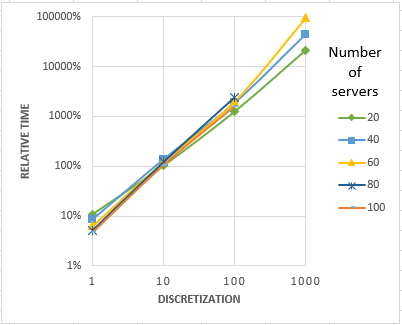}
  \caption{MCQN discretization relative time}
  \label{fig:MCQN_time}
    \end{minipage}\hfill
    \begin{minipage}{0.47\columnwidth}
        \centering
  \includegraphics[scale=0.37]{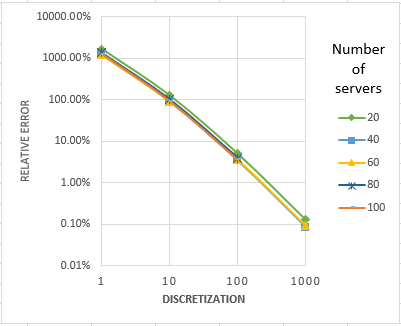}
  \caption{MCQN discretization relative error}
  \label{fig:MCQN_error}
    \end{minipage}
\end{figure}

\subsection*{Discussion}
To check our conjecture on linear empirical complexity similarly to Lemke's algorithm for LP, we compare the number of iterations $M$ (steps) to the problem dimensions that are usually expressed by the number of variables $(2K)$ and the number of constraints $(K+I)$.  Our results show that the number of steps is proportional to $2K \cdot (K+I)$ and  even decreasing with the problem size: the number of steps normalized by $2K \cdot (K+I)$ is in a range of $[1.9 \dots 3.3]\cdot 10^{-3}$ and $[4.3 \dots 7.5] \cdot 10^{-3}$ for re-entrant line and MCQN problems, respectively, where the lower values happen for larger settings in both problem classes.

%
%

\section{Summary and Future Directions}
\label{sec.impact}
To summarize, we have shown that SCLP-simplex is viable for re-entrant line and control of queueing networks problems.  It indicates that our implementation  opens new opportunities in optimizing important classes of problems, as listed in Section \ref{sec.setup}.  
For all these applications we stress the important advantage of the SCLP-simplex  in enabling us to do sensitivity analysis.  
In addition, it can be easily adjusted to  online environments, e.g., with rolling time horizon or model predictive control. Solution for a new period $[t_0, T+t_0]$ could be obtained by truncating the solution for time horizon $T$ at $t_0$ and then re-solving the problem starting from the truncated solution by increasing the time horizon from $T-t_0$ up to $T$ through its regular parametric line. During numerical experiments we found that the number of iterations of SCLP-simplex decreases exponentially with the growth of the time horizon that may be especially useful in these settings. On the other hand, in many cases the discretized LP model will need to be solved for the whole new time horizon from scratch since solutions of previous iterations may be infeasible for the new period.

There is a wide scope for further research and development of the continuous-time SCLP-simplex algorithms:
\begin{compactitem}[-]
\item
{\em Measure valued SCLP:}  while strong duality may fail in SCLP, formulation in the space of measures 
rather than the space of densities achieves strong duality, by allowing impulse controls at 0 and $T$, see \cite{shindin-weiss:14duality,shindin-weiss:15form,shindin-weiss:18simplex}.
\item
{\em Piecewise constant data:}  can be solved similarly to \cite{shindin-weiss:18simplex}.
\item
{\em Continuous fractional programming:} can be formulated as SCLP.
\item
{\em Maximum flow over time with loses and arc delays:} as formulated and discussed  in 
\cite{pullan:97arcdelay,gross-skutella:11} 
is a challenging problem for which we may be able to characterize optimal solutions.  
\item
{\em Piecewise analytic objective and right hand side:}  these models were discussed in \cite{pullan:95forms, pullan:96duality}, and it may be possible to define a simplex-type algorithm for them.
\item
{\em General continuous linear programs:} as formulated by Bellman \cite{bellman:53} seem to be of a different nature than SCLP, and present a challenging area of future research. 
\end{compactitem}

\bibliography{SCLPbibliography} 
\bibliographystyle{ieeetr}      

\end{document}